\documentclass[12pt]{amsart}

\usepackage{amssymb,amsmath,graphicx, amsthm, MnSymbol, array}
\usepackage[all]{xy}

\newtheorem{theorem}{Theorem}[section]
\newtheorem{lemma}[theorem]{Lemma}
\newtheorem{proposition}[theorem]{Proposition}

\newtheorem{claim}[theorem]{Claim}
\theoremstyle{definition}
\newtheorem{definition}[theorem]{Definition}

\newtheorem{example}[theorem]{Example}

\newtheorem{remark}[theorem]{Remark}
\newtheorem{notation}[theorem]{Notation}

\newenvironment{proofclaim}{\paragraph{\emph{Proof of the Claim}.}}{\hfill$\qed$\\}

\newcommand{\func}[1]{\operatorname{#1}}

\newcommand{\thu}{{\twoheaduparrow}}
\def\int{{\sf int}}
\newcommand\cl{{\sf cl}}

\newcommand{\dv}[1]{\mathfrak{#1}}
\newcommand\RO{\mathcal{RO}}

\newcommand\dev{{\sf DeV}}
\newcommand\deve{{\sf DeVe}}

\newcommand\Mdeve{{\sf MDeVe}}

\newcommand\C{{\sf Comp}}
\newcommand\SC{{\sf SComp}}
\newcommand\creg{{\sf CReg}}

\newcommand\KHaus{{\sf KHaus}}

\newcolumntype{x}[1]{>{\raggedright\arraybackslash}p{#1}}

\begin{document}

\title{De Vries duality for compactifications and completely regular spaces}
\author{G.~Bezhanishvili, P.~J.~Morandi, B.~Olberding}
\date{}


\begin{abstract}
De Vries duality yields a dual equivalence between the category of compact Hausdorff spaces and a category of complete Boolean algebras with a proximity relation on them, known as de Vries algebras. We extend de Vries duality to completely regular spaces by replacing the category of de Vries algebras with certain extensions of de Vries algebras. This is done by first formulating a duality between compactifications and de Vries extensions, and then specializing to the extensions that correspond to Stone-\v{C}ech compactifications.
\end{abstract}

\maketitle

\section{Introduction}

As fundamental objects of study in topology, completely regular spaces have a long and interesting history. It is a celebrated result of Tychonoff that a space is completely regular iff it is homeomorphic to a subspace of a compact Hausdorff space, and hence admits a compactification (see, e.g., \cite[Secs.~3.2 and 3.5]{Eng89}). A well-known theorem of Smirnov asserts that compactifications of a completely regular space $X$ can be described ``internally" by means of proximities on $X$ compatible with the topology on $X$ (see, e.g., \cite[Sec.~7]{NW70}). Proximities are binary relations on the powerset of $X$ satisfying certain natural axioms, including a point-separation axiom (see, e.g., \cite[Sec.~3]{NW70}). It was shown by de Vries \cite{deV62} that they can alternatively be described as binary relations on the regular open subsets of $X$ satisfying the same axioms with the notable exception that the point-separation axiom is replaced by the point-free axiom asserting that the proximity relation is approximating. On the one hand, this yields an alternative proof of Smirnov's theorem. On the other hand, it provides an algebraic description of the category $\KHaus$ of compact Hausdorff spaces by establishing that $\KHaus$ is dually equivalent to the category of complete Boolean algebras with a proximity relation on them, known as de Vries algebras.

Our goal in this article is to extend de Vries duality from the setting of compact Hausdorff spaces to that of compactifications of completely regular
spaces. This we do by introducing a category $\C$ whose objects are compactifications $e : X \rightarrow Y$, where $X$ is a completely regular space,
$Y$ is compact Hausdorff, and $e$ is a dense embedding. For compactifications $e:X \rightarrow Y$ and $e':X' \rightarrow Y'$, a morphism in this
category is given by pairs $(f,g)$ of continuous maps such that the following diagram commutes.
\[
\xymatrix@C5pc{
X \ar[r]^{e} \ar[d]_{f} & Y \ar[d]^{g} \\
X' \ar[r]_{e'} & Y'
}
\]
When $X = X'$, the compactifications $e$ and $e'$ can be isomorphic in $\C$ but not equivalent in the classical sense. (Interestingly, however,
equivalence and isomorphism do coincide for the Stone-\v{C}ech compactfication;
see Theorem~\ref{prop: equivalent to beta}.) Our interest in the category $\C$ is that
it allows more flexibility than the classical situation in that we can vary the base space $X$ of a compactification rather than just its target $Y$.
One of our main results, Theorem~\ref{thm: duality}, is that just as the category $\KHaus$ has (via de Vries duality) as its algebraic counterpart
the category $\dev$ of de Vries algebras and de Vries morphisms, the category $\C$ has  as its algebraic counterpart a category whose objects are
certain  de Vries morphisms in the category $\dev$.  Much of the work in this article involves axiomatizing these dual morphisms and showing that
the proposed framework does indeed yield a dual equivalence of categories.

To formalize this idea, we naturally rely on de Vries duality, but more is required: Given a compactification $e:X \rightarrow Y$, since $Y$ is an
object in ${\KHaus}$, de Vries duality produces a corresponding de Vries algebra $(\RO(Y),\prec)$, where $\RO(Y)$ is the complete Boolean algebra
of regular open subsets of $Y$ and $\prec$ is the canonical proximity on $\RO(Y)$
(see Section 2 for details). To deal in turn with $X$, observe that even though $X$ need not be in ${\KHaus}$, we can still view the complete
Boolean algebra $\RO(X)$ as a de Vries algebra by pulling back via $e^{-1}$ the proximity $\prec$ on $\RO(Y)$ to a proximity $\prec_Y$ on $\RO(X)$.
Thus, a first candidate for an algebraic dual to $e:X \rightarrow Y$ is the pullback map $e^{-1}:(\RO(Y),\prec) \rightarrow (\RO(X),\prec_Y)$.
However, this idea is too coarse to be useful:
the pullback map is actually a Boolean algebra isomorphism (see Lemma~\ref{lem:RO(X)}), which thus collapses information about $X$. Moreover,
the de Vries dual of $(\RO(X),\prec_Y)$ is $Y$, not $X$. Consequently, a double application of de Vries duality to the base space and its target
falls short of what we want.

These obstacles suggest a different approach is needed. Our way of dealing with this is to continue to use de Vries duality to deal with the
targets $Y$ and $Y'$ of the compactifications $e:X \rightarrow Y$ and $e':X' \rightarrow Y'$, but to work with Tarski duality for complete and
atomic Boolean algebras when dealing with the base spaces $X$ and $X'$. In other words, the algebraic dual we propose for $e:X \rightarrow Y$ is
$e^{-1}:(\RO(Y),\prec) \rightarrow (\wp(X),\subseteq)$, where $\wp(X)$ is the powerset of $X$. An additional subtlety here is that $e^{-1}$ need
not be a homomorphism of Boolean algebras, but this causes no difficulty in our setting since it is a morphism in the category of de Vries algebras.

To develop this idea into a functor, we define a category $\deve$ consisting of what we call ``de Vries extensions.'' These are 1-1 de Vries
morphisms $\alpha:\dv{A} \rightarrow \dv{B}$ of de Vries algebras $\dv{A}$ and $\dv{B}$ such that $\alpha[\dv{A}]$ is join-meet dense in
$\dv{B}$ and $\dv{B}$ is a complete and atomic Boolean algebra whose proximity is given by the partial order on $\dv{B}$ (see Section~4).
With morphisms of de Vries extensions  defined in an obvious way, we obtain the category $\deve$, and we prove in Theorem~\ref{thm: duality}
that $\KHaus$ is dually equivalent to $\deve$. Implicit in this dual equivalence is an algebraic axiomatization  of the de Vries morphisms
that arise as $e^{-1}:(\RO(Y),\prec) \rightarrow (\wp(X),\subseteq)$ for some compactification $e:X \rightarrow Y$.

Focusing next on completely regular spaces rather than their compactifications, we also obtain a duality for these spaces, and in so doing
extend de Vries duality and Tarksi duality from the category ${\KHaus}$ and the category of discrete spaces, respectively, to the category
of completely regular spaces.  This is done by first observing that the category of completely regular spaces can be identified as a full
subcategory of the category of compactifications by considering Stone-\v{C}ech compactifications. We prove that dually these correspond to
maximal de Vries extensions, thus yielding a duality for completely regular spaces that generalizes de Vries duality and Tarski duality
(see Remark~\ref{dualities remark}).

The article is organized as follows. In Section~2 we recall de Vries algebras and de Vries duality. In Section~3 we recall some classical
facts about compactifications and introduce the category $\C$ of compactifications. In Section 4, after recalling Tarski duality, we introduce
de Vries extensions and the category $\deve$ of de Vries extensions. In Section~5 we prove that $\C$ and $\deve$ are dually equivalent, thus
generalizing de Vries and Tarski dualities. Finally, in Section~6 we define maximal de Vries extensions and prove that under the duality result
of Section~5 they correspond to Stone-\v{C}ech compactifications. From this we derive that the category $\creg$ of completely regular spaces is
dually equivalent to the full subcategory $\Mdeve$ of $\deve$ consisting of maximal de Vries extensions.

In the followup paper \cite{BMO18b} we illustrate the utility of
our point of view by showing how
to obtain algebraic counterparts of normal and locally compact Hausdorff spaces in the form of de Vries extensions that are subject to additional
axioms which encode the desired topological property. As a further application, we show that a duality for locally compact Hausdorff spaces due to
Dimov \cite{Dim10} can be obtained from our perspective.

\section{De Vries duality} \label{sec: de Vries duality}

Let $\KHaus$ be the category of compact Hausdorff spaces and continuous maps. In \cite{deV62} de Vries described a category dual to $\KHaus$. The objects of this category are now known as de Vries algebras and its morphisms as de Vries morphisms. In this section we briefly review de Vries duality, and refer the interested reader to \cite[Ch.~I]{deV62} for more details.

\begin{definition}\label{def: DeV}
\begin{enumerate}
\item[]
\item A binary relation $\prec$ on a Boolean algebra $A$ is a \emph{de Vries proximity} provided it satisfies the following axioms:
\begin{enumerate}
\item[(DV1)] $1\prec 1$.
\item[(DV2)] $a\prec b$ implies $a\le b$.
\item[(DV3)] $a\le b\prec c\le d$ implies $a\prec d$.
\item[(DV4)] $a\prec b,c$ implies $a\prec b\wedge c$.
\item[(DV5)] $a\prec b$ implies $\neg b\prec \neg a$.
\item[(DV6)] $a\prec b$ implies there is $c$ such that $a\prec c\prec b$.
\item[(DV7)] $b\ne 0$ implies there is $a \ne 0$ such that $a\prec b$.
\end{enumerate}
\item A \emph{de Vries algebra} is a pair $\mathfrak A=(A,\prec)$, where $A$ is a complete Boolean algebra and $\prec$ is a de Vries proximity on $A$.
\end{enumerate}
\end{definition}

\begin{remark}
It follows from the definition that $0 \prec 0$ and $a,b \prec c$ implies $a \vee b \prec c$.
In addition, the axiom (DV7) can equivalently be stated as $b = \bigvee\{a\in A \mid a\prec b\}$, and asserts that the proximity $\prec$ is approximating.
\end{remark}

\begin{definition}\label{def: DeV morphism}
A map $\rho: \dv{A} \to \dv{B}$ between de Vries algebras is a \emph{de Vries morphism} provided
\begin{enumerate}
\item[(M1)] $\rho(0)=0$.
\item[(M2)] $\rho(a\wedge b)=\rho(a)\wedge\rho(b)$.
\item[(M3)] $a\prec b$ implies $\neg\rho(\neg a)\prec\rho(b)$.
\item[(M4)] $\rho(b) = \bigvee \{\rho(a)\mid a\prec b\}$.
\end{enumerate}
\end{definition}

\begin{remark} \label{rem: morphism properties}
Each de Vries morphism $\rho$ satisfies:
\begin{enumerate}
\item $\rho(1) = 1$;
\item $\rho(\neg a)\le\neg\rho(a)$;
\item $\rho(a)\le\neg\rho\neg(a)$;
\item $a \prec b$ implies $\rho(a) \prec \rho(b)$.
\end{enumerate}
Thus, $\rho$ is a bounded meet-semilattice homomorphism that preserves $\prec$. However, $\rho$ does not preserve join or negation in general.
Axiom (M3) is equivalent to the following axiom: $a_1 \prec b_1$ and $a_2\prec b_2$ imply $\rho(a_1 \vee a_2) \prec \rho(b_1) \vee \rho(b_2)$
\cite[Lem.~2.2]{Bez12}.
\end{remark}

Function composition of two de Vries morphisms need not be a de Vries morphism because it need not satisfy (M4). To repair this, the composition of two de Vries morphisms $\rho_1:\dv{A}_1\to \dv{A}_2$ and $\rho_2: \dv{A}_2 \to \dv{A}_3$ is defined by
\[
(\rho_2 \star \rho_1)(b)=\bigvee\{\rho_2\rho_1(a)\mid  a\prec b\}.
\]
With this composition, de Vries algebras and de Vries morphisms form a category.

\begin{definition}
Let $\dev$ be the category of de Vries algebras and de Vries morphisms, where composition is defined as above.
\end{definition}

\begin{remark} \label{rem: composition}
\begin{enumerate}
\item[]
\item As pointed out in Remark~\ref{rem: morphism properties}, a de Vries morphism is not a Boolean homomorphism in general. Nevertheless, isomorphisms in $\dev$ are Boolean isomorphisms which preserve and
reflect proximity.
\item While composition in $\dev$ is not function composition in general, if $\rho_1 : \dv{A}_1 \to \dv{A}_2$ and $\rho_2 : \dv{A}_2 \to \dv{A}_3$
are de Vries morphisms such that $\rho_2$ is a complete Boolean homomorphism, then it follows from the definition of $\star$ that
$\rho_2 \star \rho_1 = \rho_2 \circ \rho_1$. We will use this throughout the paper.
\end{enumerate}
\end{remark}

By de Vries duality, $\KHaus$ is dually equivalent to $\dev$. We briefly describe the functors that yield this dual equivalence. For $X \in \KHaus$, let $X^* = (\mathcal{RO}(X),\prec)$, where $\mathcal{RO}(X)$ is the complete Boolean algebra of regular open subsets of $X$ and $\prec$ is the
canonical proximity on $\RO(X)$. The Boolean operations on $\mathcal{RO}(X)$ are given by
\begin{itemize}
\item $\bigvee_i U_i = {\int}\left({\cl}\left(\bigcup_i U_i\right)\right)$,
\item $\bigwedge_i U_i = {\int}\left(\bigcap_i U_i\right)$,
\item $\lnot U = {\int}(X\backslash U)$;
\end{itemize}
and the canonical proximity is given by
\[
U\prec V \mbox{ iff } {\cl}(U)\subseteq V.
\]
Then $X^* \in \dev$. If $f:X\to Y$ is a continuous map, then $f^*: Y^*\to X^*$ is given by
\[
f^*(U)={\int}\left({\cl}\left(f^{-1}(U)\right)\right).
\]
We have that $f^*$ is a de Vries morphism. This yields a contravariant functor $(-)^* : \KHaus \to \dev$.

To define a contravariant functor $(-)_* : \dev \to \KHaus$, let $\dv{A} \in \dev$. For $S\subseteq \dv{A}$, let $\thu S = \{a\in \dv{A} \mid  b\prec a \mbox{ for some } b\in S\}$. We call a filter $F$ of $\dv{A}$ \emph{round} if $\thu F=F$. An \emph{end} is a maximal proper round filter. Let $Y_\dv{A}$ be the set of ends of $\dv{A}$. For $a\in \dv{A}$, let $$\zeta_\dv{A}(a)=\{x\in Y_\dv{A} \mid a\in x\}.$$ Define a topology on $Y_\dv{A}$ by letting
$$
\zeta_\dv{A}[\dv{A}]=\{\zeta_\dv{A}(a) \mid a\in \dv{A}\}
$$
be a basis for the topology. The space $Y_\dv{A}$ is compact Hausdorff and $\zeta_\dv{A} : \dv{A} \to \RO(Y_\dv{A})$ is a de Vries isomorphism.

If $\rho : \dv{A} \to \dv{A}'$ is a de Vries morphism, we define $\rho_* : Y_\dv{A'} \to Y_\dv{A}$ by $\rho_*(y) = \thu \rho^{-1}(y)$. Then $\rho_*$ is continuous. This yields a contravariant functor $(-)_* : \dev \to \KHaus$.

As we already pointed out, $\zeta_{\dv{A}} : \dv{A} \to (\dv{A}_*)^* $ is a de Vries isomorphism. Moreover, $\xi_X : X \to (X^*)_*$ is a homeomorphism, where $\xi(x) = \{ U \in X^* \mid x \in U\}$. Therefore, $\zeta : 1_\dev \to (-)^* \circ (-)_*$ and $\xi : 1_\KHaus \to (-)_* \circ (-)^*$ are natural isomorphisms. Thus, we arrive at de Vries duality.

\begin{theorem}[de Vries duality]
The functors $(-)^*$ and $(-)_*$ give a dual equivalence between $\KHaus$ and $\dev$.
\end{theorem}

\section{Compactifications} \label{sec: compactifications}

Our goal is to extend de Vries duality from compact Hausdorff spaces to completely regular spaces. This we do by utilizing the theory of compactifications of completely regular spaces (see, e.g., \cite[Ch.~3.5]{Eng89}). We recall that a \emph{compactification} of a completely regular space $X$ is a pair $(Y,e)$, where $Y$ is a compact Hausdorff space and $e:X\to Y$ is an embedding such that the image $e(X)$ is dense in $Y$.

Suppose that $e : X \to Y$ and $e' : X \to Y'$ are compactifications. As usual, we write $e \le e'$ provided there is a continuous map $f : Y' \to Y$
with $f \circ e' = e$.

\[
\xymatrix{
X \ar[r]^{e'} \ar[rd]_{e} & Y' \ar[d]^{f} \\
& Y
}
\]
The relation $\le$ is reflexive and transitive. Two compactifications $e$ and $e'$ are \emph{equivalent} if $e \le e'$ and $e' \le e$. It is well known that $e$ and $e'$ are equivalent iff there is a homeomorphism $f : Y'\to Y$ with $f \circ e' = e$. The equivalence classes of compactifications of $X$ form a poset whose largest element is the Stone-\v{C}ech compactification $s : X \to \beta X$.

In the classical setting, one considers compactifications of a fixed base space $X$. For our purposes we need to vary the base space. The proof of the following proposition is straightforward.

\begin{proposition}
There is a category $\C$ whose objects are compactifications $e:X\to Y$ and whose morphisms are pairs $(f,g)$ of continuous maps such that
the following diagram commutes.
\[
\xymatrix@C5pc{
X \ar[r]^{e} \ar[d]_{f} & Y \ar[d]^{g} \\
X' \ar[r]_{e'} & Y'
}
\]
The composition of two morphisms $(f_1,g_1)$ and $(f_2,g_2)$ is defined to be $(f_2\circ f_1, g_2\circ g_1)$.
\[
\xymatrix@C5pc{
X_1 \ar[r]^{e_1} \ar@/_1.5pc/[dd]_{f_2\circ f_1} \ar[d]^{f_1} & Y_1 \ar[d]_{g_1} \ar@/^1.5pc/[dd]^{g_2\circ g_1}\\
X_2 \ar[r]^{e_2} \ar[d]^{f_2} & Y_2 \ar[d]_{g_2} \\
X_3 \ar[r]_{e_3} & Y_3
}
\]
\end{proposition}

It is straightforward to see that a morphism $(f,g)$ in $\C$ is an isomorphism iff both $f$ and $g$ are homeomorphisms. From this it follows that equivalent compactifications of $X$ are isomorphic in $\C$. The converse is not true in general, as the following example shows.

\begin{example} \label{ex: iso isn't equivalence}
Let $X$ be the set $\mathbb{N}$ of natural numbers equipped with the discrete topology, and let $e : X \to Y$ be the two-point compactification of $X$ where $Y$ is the disjoint union of the one-point compactifications of $2\mathbb{N}$ and $2\mathbb{N}+1$ and $e$ is the inclusion map. Thus, $Y = \mathbb{N} \cup \{\infty_e, \infty_o\}$ where $\cl_Y(2\mathbb{N}) = 2\mathbb{N} \cup \{\infty_e\}$ and $\cl_Y(2\mathbb{N}+1) = (2\mathbb{N}+1) \cup \{\infty_o\}$.

Define $f: X \to X$ as follows: Let $f$ be the identity on $4\mathbb{N}$ and $4\mathbb{N}+1$, and for $n \ge 0$ set $f(4n+2) = 4n+3$ and $f(4n+3) = 4n+2$ (see the diagram).
\begin{center}
\includegraphics[width=4in]{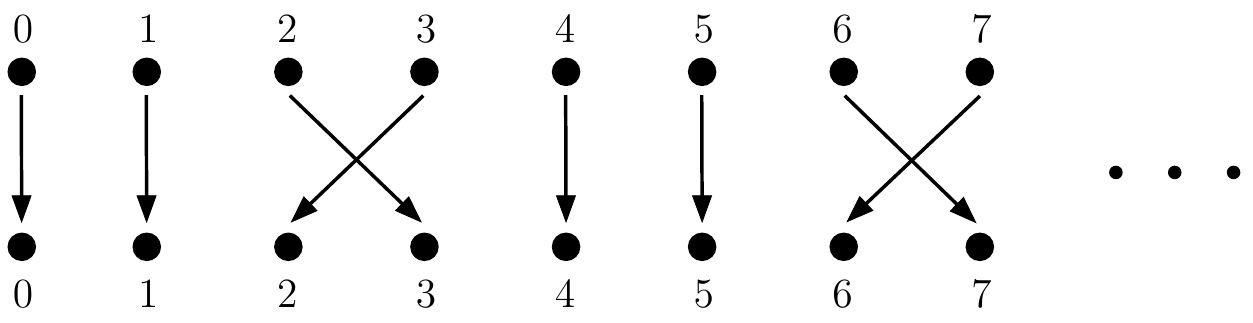}
\end{center}
Then $f$ is a homeomorphism with $f(2\mathbb{N}) = 4\mathbb{N} \cup (4\mathbb{N}+3)$ and $f(2\mathbb{N}+1) = (4\mathbb{N}+1) \cup (4\mathbb{N}+2)$.

Let $e':X\to Y'$ be the two-point compactification of $X$ where $Y'$ is the disjoint union of the one-point compactifications of $f(2\mathbb{N})$ and $f(2\mathbb{N}+1)$ and $e' : X \to Y'$ is the inclusion map. Then $Y' = \mathbb{N} \cup \{\infty_1, \infty_2\}$ where $\infty_1$ is a limit point of $f(2\mathbb{N})$ and $\infty_2$ is a limit point of $f(2\mathbb{N}+1)$.

Define $g : Y \to Y'$ by $g(n) = f(n)$ for each $n \in \mathbb{N}$, $g(\infty_e) = \infty_1$, and $g(\infty_o) = \infty_2$.
A straightforward
argument shows that $g$ is a homeomorphism and the following diagram is commutative.
\[
\xymatrix{
X \ar[r]^{e} \ar[d]_{f} & Y \ar[d]^{g} \\
X \ar[r]_{e'} & Y'
}
\]
Therefore, $(f, g)$ is an isomorphism in $\C$. On the other hand, we show that $e$ and $e'$ are not equivalent.
Suppose that there is a homeomorphism $h : Y \to Y'$ with $h \circ e = e'$.
\[
\xymatrix{
X \ar[r]^{e} \ar[dr]_{e'} & Y \ar[d]^{h} \\
& Y'
}
\]
Let $A = 4\mathbb{N} \cup (4\mathbb{N}+3) = f(2\mathbb{N})$. Then $\cl_Y(A) = A \cup \{\infty_e, \infty_o\}$ because $\infty_e$ is a limit point of $4\mathbb{N}$ and $\infty_o$ is a limit point of $4\mathbb{N}+3$. Since $h$ is a homeomorphism extending the identity on $X$, we have $h(\cl_Y(A)) = \cl_{Y'}(A)$. However, $A = f(2\mathbb{N})$, so $\cl_{Y'}(A) = A \cup \{\infty_1\}$ while $h(\cl_Y(A)) = A \cup \{\infty_1, \infty_2\}$. This contradiction shows that there is no homeomorphism $h$ with $h \circ e = e'$, and hence $e$ and $e'$ are not equivalent.
\end{example}

On the other hand, if a compactification is equivalent to the Stone-\v{C}ech compactification, then the two compactifications are isomorphic in $\C$.

\begin{theorem} \label{prop: equivalent to beta}
Let $e : X \to Y$ be a compactification. If $e$ is isomorphic to $s : X \to \beta X$ in $\C$, then $e$ is equivalent to $s$.
\end{theorem}

\begin{proof}
By hypothesis, there is an isomorphism $(f,g)$ in $\C$ yielding the following commutative diagram.
\[
\xymatrix{
X \ar[r]^{e} \ar[d]_{f} & Y \ar[d]^{g} \\
X \ar[r]_{s} & \beta X
}
\]
This means that both $f$ and $g$ are homeomorphisms. By the universal mapping property for $\beta$ (see, e.g., \cite[Thm.~3.6.1]{Eng89}), there is
$\beta(f) : \beta X \to \beta X$ with $\beta(f) \circ s = s \circ f$.
\[
\xymatrix{
X \ar[r]^s \ar[d]_{f} & \beta X \ar[d]^{\beta(f)} \\
X \ar[r]_s & \beta X
}
\]
Note that $\beta(f)$ is a homeomorphism since $\beta$ is a functor (see, e.g., \cite[Sec.~IV.2.1]{Joh82}) and $f$ is a homeomorphism. Let $h = g^{-1} \circ \beta(f) : \beta X \to Y$. Then $h : \beta X \to Y$ is a homeomorphism.
\[
\xymatrix{
X \ar[r]^{s} \ar[dr]_{e} & \beta X \ar[d]^{h}\\
& Y
}
\]
Moreover, $h \circ s = g^{-1} \circ \beta(f) \circ s = g^{-1} \circ s \circ f = g^{-1}\circ (g \circ e) = e$. Thus, $e : X\to Y$ is equivalent to $s : X \to \beta X$.
\end{proof}

\section{De Vries extensions} \label{sec: de Vries}

In this section we introduce the concept of a de Vries extension, which is the main concept of this article. To motivate the definition,
let $e : X \to Y$ be a compactification of a completely regular space $X$. Since $Y$ is compact Hausdorff, by de Vries duality, we can
associate with $Y$ the de Vries algebra $(\RO(Y),\prec)$. Although $X$ is not necessarily compact, it would be natural to work with the
complete Boolean algebra $\RO(X)$ of regular open subsets of $X$. However, as the next lemma shows, $\RO(X)$ is isomorphic to $\RO(Y)$.

\begin{lemma} \label{lem:RO(X)}
Let $e : X \to Y$ be a compactification. Then $e^{-1} : \RO(Y) \to \RO(X)$ is a Boolean isomorphism.
\end{lemma}

\begin{proof}
To simplify notation, we view $X$ as a dense subspace of $Y$. Then the map $e^{-1}$ sends $V$ to $V \cap X$. Since $X$ is dense in $Y$, if $V$ is open in $Y$, then we have ${\cl}_Y(V \cap X) = {\cl}_Y(V)$.

\begin{claim} \label{claim: RO}
If $U $ is regular open in $X$, then $U = \int_Y(\cl_Y(U))\cap X$.
\end{claim}

\begin{proofclaim}
Set $V = {\int}_Y({\cl}_Y(U))$. Since $U$ is open in $X$, we may write $U = W \cap X$ for some open set $W$ in $Y$. As $X$ is dense in $Y$,
we have ${\cl}_Y(U) = {\cl}_Y(W)$. Therefore, $W \subseteq {\int}_Y({\cl}_Y(U)) = V$. This yields $U \subseteq V \cap X$. On the other hand,
${\cl}_X(U) = {\cl}_Y(U) \cap X \supseteq V \cap X$. Thus, $V \cap X \subseteq {\int}_X({\cl}_X(U)) = U$. Consequently, $V \cap X = U$,
and so the claim is verified.
\end{proofclaim}

We next show that if $V \in \RO(Y)$, then $V\cap X \in \RO(X)$. For, let $U = {\int}_X({\cl}_X(V\cap X))$. Then $U \in \RO(X)$ and $V \cap X \subseteq U$. We have
\[
U = {\int}_X({\cl}_X(V\cap X)) \subseteq {\cl}_X(V \cap X) \subseteq {\cl}_Y(V).
\]
Therefore, ${\cl}_Y(U) \subseteq {\cl}_Y(V)$. Thus, $W := {\int}_Y({\cl}_Y(U)) \subseteq {\int}_Y({\cl}_Y(V)) = V$. Now, $W \cap X = U$ by Claim~\ref{claim: RO}. Therefore, $V \cap X \subseteq U = W \cap X \subseteq V \cap X$, and so $U = V \cap X$, yielding that $V \cap X \in \RO(X)$. Thus, $e^{-1} : \RO(Y) \to \RO(X)$ is well defined, and it is onto by Claim~\ref{claim: RO}.

To see $e^{-1}$ is 1-1, let $V$ be regular open in $Y$. Then $V = {\int}_Y({\cl}_Y(V)) = {\int}_Y({\cl}_Y(V\cap X))$. Therefore, if $V,W \in \RO(Y)$ with $V \cap X = W \cap X$, then this formula shows $V = W$. Thus, $e^{-1}$ is a bijection. It clearly preserves order, and this formula shows that it also reflects order. Consequently, $e^{-1}$ is an order isomorphism, hence a Boolean isomorphism.
\end{proof}

Instead of working with $\RO(X)$, we will work with the powerset $\wp(X)$ of $X$, and will utilize Tarski duality between the category $\sf CABA$ of complete and atomic Boolean algebras with complete Boolean homomorphisms and the category $\sf Set$ of sets and functions. If $X$ is a set, then $\wp(X)$ is a complete and atomic Boolean algebra, and if $f : X \to Y$ is a function, then $f^{-1} : \wp(Y) \to \wp(X)$ is a complete Boolean homomorphism. This yields a contravariant functor $ {\sf Set} \to {\sf CABA}$. Going backwards, for a complete and atomic Boolean algebra $\dv{B}$, let $X_\dv{B}$ be the set of atoms of $\dv{B}$, and for a complete Boolean homomorphism $\sigma : \dv{B}_1 \to \dv{B}_2$, let $\sigma_+ : X_{\dv{B}_2} \to X_{\dv{B}_1}$ be given by $\sigma_+(x) = \bigwedge \{ b \in \dv{B}_1 \mid x \le \sigma(b) \}$. It is well known that $\sigma_+$ is a well-defined function, yielding a contravariant functor ${\sf CABA} \to {\sf Set}$. For each set $X$, we have a natural isomorphism $\eta_X : X \to X_{\wp(X)}$,
given by $\eta_X(x) = \{x\}$ for each $x \in X$; and for each $\dv{B} \in \sf CABA$, we have a natural isomorphism $\vartheta_\dv{B} : \dv{B} \to \wp(X_\dv{B})$, given by $\vartheta_\dv{B}(b) = \{ x \in X_\dv{B} \mid x \le b\}$.

\begin{theorem}
[Tarski duality] The functors described above give a dual equivalence between $\sf Set$ and $\sf CABA$.
\end{theorem}

For each complete Boolean algebra $A$, the pair $\dv{A}=(A,\le)$ is a de Vries algebra of a special kind. Such de Vries algebras
are called \emph{extremally disconnected} in \cite{Bez10} since they correspond to extremally disconnected compact Hausdorff spaces.
Therefore, the pair $(\wp(X), \subseteq)$ is an extremally disconnected de Vries algebra which in addition is an atomic Boolean algebra.
Thus, each compactification $e:X\to Y$ gives rise to the de Vries algebra $(\RO(Y),\prec)$, the atomic extremally disconnected de Vries algebra
$(\wp(X),\subseteq)$, and the map $e^{-1}:\RO(Y)\to\wp(X)$.


\begin{notation}
Let $e : X \to Y$ be a compactification. In the remainder of this paper we write $\RO(Y)$ for the de Vries algebra $(\RO(Y), \prec)$, and $\wp(X)$ for the atomic extremally disconnected de Vries algebra $(\wp(X), \subseteq)$.
\end{notation}

\begin{theorem} \label{prop: extension of a compactification}
Let $e : X \to Y$ be a compactification. Then $e^{-1} : \RO(Y) \to \wp(X)$ is a \emph{1-1} de Vries morphism such that each element of $\wp(X)$ is a join of meets from $e^{-1}[\RO(Y)]$.
\end{theorem}

\begin{proof}
We first show that $e^{-1}$ is a de Vries morphism. Obviously $e^{-1}(\varnothing) = \varnothing$, so (M1) holds. Since binary meet in $\RO(Y)$ is intersection, it is clear that (M2) holds. Suppose $U,V \in \RO(Y)$ with $U \prec V$. Then ${\cl_Y}(U)\subseteq V$, so $Y\setminus V \subseteq Y\setminus {\cl_Y}(U)={\int_Y}(Y\setminus U)$. Therefore, $e^{-1}(Y\setminus V)\subseteq e^{-1}{\int_Y}(Y\setminus U)$, and so $X\setminus e^{-1}{\int_Y}(Y\setminus U)\subseteq X\setminus e^{-1}(Y\setminus V)=e^{-1}(V)$. Thus, (M3) holds. If $U \in \RO(Y)$, then since $Y$ is compact Hausdorff, $U=\bigcup\{V\in\RO(Y) \mid {\cl_Y}(V)\subseteq U\}$, so $e^{-1}(U)=\bigcup\{e^{-1}(V) \mid V\prec U\}$. Consequently, (M4) holds, and hence $e^{-1}$ is a de Vries morphism. By Lemma~\ref{lem:RO(X)}, $e^{-1}$ is 1-1. Finally, the atoms of $\wp(X)$ are the singletons. Because $X$ is regular, if $x \in X$, then $\{x\}$ is the intersection of all regular open subsets of $X$ containing $x$. By Lemma~\ref{lem:RO(X)}, they are of the form $e^{-1}(U)$ with $U\in\RO(Y)$. Since each element in $\wp(X)$ is a union of singletons, $\wp(X)$ is a join of meets from $e^{-1}[\RO(Y)]$.
\end{proof}

Theorem~\ref{prop: extension of a compactification} motivates Definition~\ref{def: DeVe} below. In order to formulate it succinctly, we require the following definition.

\begin{definition}
Let $\alpha : \dv{A} \to \dv{B}$ be a de Vries morphism.
\begin{enumerate}
\item We say $\alpha[\dv{A}]$ is join-meet dense in $\dv{B}$ if each element of $\dv{B}$ is a join of meets from $\alpha[\dv{A}]$.
\item We say $\alpha[\dv{A}]$ is meet-join dense in $\dv{B}$ if each element of $\dv{B}$ is a meet of joins from $\alpha[\dv{A}]$.
\end{enumerate}
\end{definition}

\begin{remark} \label{rem: jm density}
Let $\alpha : \dv{A} \to \dv{B}$ be a de Vries morphism with $\dv{B}$ an atomic extremally disconnected de Vries algebra.
\begin{enumerate}
\item $\alpha[\dv{A}]$ is join-meet dense in $\dv{B}$ iff each atom of $\dv{B}$ is a meet from $\alpha[\dv{B}]$.
\item $\alpha[\dv{A}]$ is meet-join dense in $\dv{B}$ iff each coatom of $\dv{B}$ is a join from $\alpha[\dv{B}]$.
\item $\alpha[\dv{A}]$ is join-meet dense in $\dv{B}$ iff $\alpha[\dv{A}]$ is meet-join dense in $\dv{B}$. To see the left to right implication, suppose that $\alpha[\dv{A}]$ is join-meet dense in $\dv{B}$. Let $x \in \dv{B}$ be a coatom. Then $\lnot x$ is an atom, so $\lnot x = \bigwedge \{ \alpha(c) \mid c \in \dv{A}, \lnot x \le \alpha(c)\}$ by (1). This yields $x = \bigvee \{ \lnot\alpha(c) \mid c \in \dv{A}, \lnot x \le \alpha(c)\}$. Let $c \in \dv{A}$ with $\lnot x \le \alpha(c)$. Since $\alpha$ is a de Vries morphism, $\alpha(c) = \bigvee \{ \alpha(a) \mid a \prec c\}$. Because $\lnot x$ is an atom, there is $a \prec c$ with $\lnot x \le \alpha(a)$. We have $\alpha(a) \le \lnot\alpha(\lnot a) \le \alpha(c)$, so $\lnot\alpha(c) \le \alpha(\lnot a)$. Also, since $\lnot x \le \alpha(a)$, we see that $\alpha(\lnot a) \le \lnot\alpha(a) \le x$. Thus, $x = \bigvee \{ \alpha(\lnot a) \mid \lnot x \le \alpha(a) \}$, and so $x$ is a join from $\alpha[\dv{A}]$. Applying (2) then gives that $\alpha[\dv{A}]$ is meet-join dense in $\dv{B}$. The right to left implication follows from a similar argument.
\end{enumerate}
\end{remark}

We are ready for the main definition of this article.

\begin{definition}\label{def: DeVe}
Let $\dv{A}$ be a de Vries algebra and $\dv{B}=(B,\le)$ an atomic extremally disconnected de Vries algebra. A \emph{de Vries extension} is a 1-1 de Vries morphism  $\alpha: \dv{A}\to \dv{B}$ such that $\alpha[\dv{A}]$ is join-meet dense in $\dv{B}$.
\end{definition}

\begin{proposition} \label{prop: deve a category}
There is a category $\deve$ whose objects are de Vries extensions and whose morphisms are pairs $(\rho, \sigma)$, where $\rho$ is a de Vries morphism, $\sigma$ is a complete Boolean homomorphism, and $\sigma\circ\alpha = \alpha' \star \rho$.
\[
\xymatrix@C5pc{
\dv{A} \ar[r]^{\alpha} \ar[d]_{\rho} & \dv{B} \ar[d]^{\sigma} \\
\dv{A'} \ar[r]_{\alpha'} & \dv{B'}
}
\]
The composition of two morphisms $(\rho_1, \sigma_1)$ and $(\rho_2, \sigma_2)$ is defined to be $(\rho_2\star \rho_1, \sigma_2\circ \sigma_1)$.
\[
\xymatrix@C5pc{
\dv{A_1} \ar@/_1.5pc/[dd]_{\rho_2\star\rho_1}\ar[r]^{\alpha_1} \ar[d]^{\rho_1} & \dv{B}_1 \ar[d]_{\sigma_1} \ar@/^1.5pc/[dd]^{\sigma_2\circ\sigma_1} \\
\dv{A}_2 \ar[r]^{\alpha_2} \ar[d]^{\rho_2} & \dv{B}_2 \ar[d]_{\sigma_2} \\
\dv{A}_3 \ar[r]_{\alpha_3} & \dv{B}_3
}
\]
\end{proposition}

\begin{proof}
Since both $\star$ and $\circ$ are associative operations, it follows that composition in $\deve$ is a well-defined associative operation, and so it is straightforward to see that $\deve$ is a category.
\end{proof}

\begin{remark} \label{rem: Fedorchuk}
Let $(\rho, \sigma)$ be a morphism in $\deve$ between de Vries extensions $\alpha : \dv{A} \to \dv{B}$ and $\alpha' : \dv{A}' \to \dv{B}'$. 
By Remark~\ref{rem: composition}(2), $\sigma \star \alpha = \sigma \circ \alpha$. Thus, the equation $\sigma \circ \alpha = \alpha' \star \rho$ 
says that the first diagram in the statement of Proposition~\ref{prop: deve a category} is a commutative diagram in $\dev$. The complete Boolean 
homomorphism $\sigma$ is an instance of a special type of de Vries morphism studied by Fedorchuk (see \cite{Fed74} or \cite[Sec.~3]{Bez10}).
\end{remark}

\section{Dual equivalence of $\C$ and $\deve$} \label{sec: duality}

This section is dedicated to proving that $\C$ and $\deve$ are dually equivalent.  We begin by defining a functor ${\sf E} : \C \to \deve$ as follows. If $e : X \to Y$ is an object in $\C$, we define its image under $\sf E$ as the de Vries extension $e^{-1} : \RO(Y) \to \wp(X)$. For a morphism $(f,g)$ in $\C$
\[
\xymatrix@C5pc{
X \ar[r]^{e} \ar[d]_{f} & Y \ar[d]^{g} \\
X' \ar[r]_{e'} & Y'
}
\]
we define ${\sf E}(f,g)$ as the pair $(g^*,f^{-1})$.
\[
\xymatrix@C4pc{
\RO(Y') \ar[d]_{g^*} \ar[r]^{(e')^{-1}} & \wp(X') \ar[d]^{f^{-1}}\\
\RO(Y) \ar[r]_{e^{-1}} & \wp(X)
}
\]

\begin{proposition} \label{prop: functor E}
${\sf E} : \C \to \deve$ is a contravariant functor.
\end{proposition}

\begin{proof}
By Theorem~\ref{prop: extension of a compactification}, for each object $e : X \to Y$ in $\C$, we have ${\sf E}(e)$ is an object in $\deve$. Let $(f,g)$ be a morphism in $\C$. By de Vries duality, $g^*$ is a de Vries morphism, and it is clear that $f^{-1}$ is a complete Boolean homomorphism. To see that the diagram
\[
\xymatrix@C4pc{
\RO(Y') \ar[d]_{g^*} \ar[r]^{(e')^{-1}} & \wp(X') \ar[d]^{f^{-1}}\\
\RO(Y) \ar[r]_{e^{-1}} & \wp(X)
}
\]
commutes, let $U\in\RO(Y')$. Since $f^{-1}$ is a complete Boolean homomorphism, $f^{-1} \star (e')^{-1} = f^{-1} \circ (e')^{-1}$ (see Remark~\ref{rem: Fedorchuk}). As $e'\circ f=g\circ e$, we have
\[
(f^{-1}\star(e')^{-1})(U)=f^{-1}((e')^{-1}(U)) = (e'\circ f)^{-1}(U)=(g\circ e)^{-1}(U)=e^{-1}g^{-1}(U).
\]
On the other hand, $(e^{-1}\star g^*)(U)=\bigcup\{e^{-1}g^*(V) \mid V\prec U\}$. From $V \prec U$ it follows that $\cl(V) \subseteq U$, so
$\int(e^{-1}g^{-1}(\cl(V))) \subseteq e^{-1}g^{-1}(U)$. Because $e$ and $g$ are continuous, $e^{-1}g^*(V) = e^{-1}(\int(\cl(g^{-1}(V))))$
is contained in $\int (e^{-1}g^{-1}(\cl(V)))$. Thus, $e^{-1}g^*(V)\subseteq e^{-1}g^{-1}(U)$, so $(e^{-1}\star g^*)(U)\subseteq e^{-1}g^{-1}(U)$.
Conversely, $x\in e^{-1}g^{-1}(U)$ implies $ge(x)\in U$, so there is $V\prec U$ with $ge(x)\in V$. Therefore,
$x \in e^{-1}g^{-1}(V)\subseteq e^{-1}(\int(\cl(g^{-1}(V))))=e^{-1}g^*(V)$, and so $e^{-1}g^{-1}(U)\subseteq (e^{-1}\star g^*)(U)$.
This shows that $e^{-1}g^{-1}(U) = (e^{-1}\star g^*)(U)$, and so $(f^{-1}\star(e')^{-1})(U) = (e^{-1}\star g^*)(U)$. Thus, the diagram commutes,
and hence ${\sf E}(f,g)$ is a morphism in $\deve$.

It is elementary to see that $\sf E$ sends identity morphisms to identity morphisms. Given the composable morphisms in $\C$ in the following diagram
\[
\xymatrix@C5pc{
X_1 \ar[r]^{e_1} \ar@/_1.5pc/[dd]_{f} \ar[d]^{f_1} & Y_1 \ar[d]_{g_1} \ar@/^1.5pc/[dd]^{g}\\
X_2 \ar[r]^{e_2} \ar[d]^{f_2} & Y_2 \ar[d]_{g_2} \\
X_3 \ar[r]_{e_3} & Y_3
}
\]
where $f = f_2 \circ f_1$ and $g = g_2 \circ g_1$, we obtain the diagram
\[
\xymatrix@C4pc{
\RO(Y_3) \ar@/_2pc/[dd]_{g^*} \ar[d]^{g_2^*} \ar[r]^{{e_3}^{-1}} & \wp(X_3) \ar[d]_{f_2^{-1}} \ar@/^2pc/[dd]^{f^{-1}} \\
\RO(Y_2) \ar[d]^{g_1^*} \ar[r]^{{e_2}^{-1}} & \wp(X_2) \ar[d]_{f_1^{-1}}\\
\RO(Y_1) \ar[r]_{e_1^{-1}} & \wp(X_1)
}
\]
Since $f=f_2 \circ f_1$, we have $f^{-1} = f_1^{-1} \circ f_2^{-1}$.
Also, by de Vries duality, $g^* = g_1^* \star g_2^*$. Therefore, the diagram commutes, and so
${\sf E}((f_2,g_2) \circ (f_1,g_1)) = {\sf E}(f_1,g_1) \circ {\sf E}(f_2,g_2)$. Thus, $\sf E$ is a contravariant functor.
\end{proof}

We next wish to define a functor ${\sf C}:{\deve}\to\C$, which requires some preparation. If $\dv{B}$ is an extremally disconnected de Vries algebra,
then the proximity is $\le$, so ends of $\dv{B}$ are simply ultrafilters of $\dv{B}$, and hence $Y_{\dv{B}}$ is the Stone space of $\dv{B}$.

\begin{definition}
For an atomic extremally disconnected de Vries algebra $\dv{B}$, let $X_{\dv{B}}$ be the set of atoms of $\dv{B}$. We identify an atom $b$ with the principal ultrafilter ${\uparrow}b$ of $\dv{B}$, and view $X_{\dv{B}}$ as a subset of $Y_{\dv{B}}$.
\end{definition}

\begin{lemma}\label{lem: atoms}
Suppose $\alpha : \dv{A} \to \dv{B}$ is a de Vries morphism with $\dv{B}$ an atomic extremally disconnected de Vries algebra. For each $a\in \dv{A}$ and each atom $b \in \dv{B}$, the following are equivalent.
\begin{enumerate}
\item $b \le \alpha(a)$.
\item $a \in \alpha_*({\uparrow}b)$.
\item ${\uparrow}b \in \alpha_*^{-1} \zeta(a)$.
\end{enumerate}
\end{lemma}

\begin{proof}
(1)$\Rightarrow$(2). Suppose that $b \le \alpha(a)$. Since $\alpha$ is a de Vries morphism, $\alpha(a)=\bigvee\{\alpha(x)\mid x\prec a\}$, so
$b\le \bigvee\{\alpha(x)\mid x\prec a\}$. As $b$ is an atom,  there is $x$ with $x\prec a$ and $b\le\alpha(x)$. Thus, $a \in \alpha_*({\uparrow}b)$.

(2)$\Rightarrow$(3). If $a  \in \alpha_*({\uparrow}b)$, then $\alpha_*({\uparrow}b) \in \zeta(a)$, so ${\uparrow}b \in \alpha_*^{-1}\zeta(a)$.

(3)$\Rightarrow$(1). If ${\uparrow}b \in \alpha_*^{-1}\zeta (a)$, then $a \in \alpha_*({\uparrow}b) = \thu \alpha^{-1}({\uparrow}b)$, so there is $x$ with $x \prec a$ and $b \le \alpha(x)$. Therefore, $b \le \alpha(x) \le \alpha(a)$.
\end{proof}

\begin{lemma} \label{lem:char of jm density}
Let $\dv{A}$ and $\dv{B}$ be de Vries algebras, with $\dv{B}$ atomic and extremally disconnected, and let $\alpha : \dv{A} \to \dv{B}$ be a \emph{1-1} de Vries morphism. Then $\alpha$ is a de Vries extension iff the restriction of $\alpha_* : Y_{\dv{B}} \to Y_\dv{A}$ to $X_{\dv{B}}$ is \emph{1-1}.
\end{lemma}

\begin{proof}
Suppose that $\alpha$ is a de Vries extension. Let $b\ne c$ be atoms of $\dv{B}$. By Remark~\ref{rem: jm density}(1), each of $b$ and $c$ are meets from $\alpha[\dv{A}]$. Therefore, there is $a \in \dv{A}$ with $b \le \alpha(a)$ and $c \not\le \alpha(a)$. By Lemma~\ref{lem: atoms}, $a \in \alpha_*({\uparrow}b)$ and $a \notin \alpha_*({\uparrow}c)$. Thus, $\alpha_*$ is 1-1 on $X_{\dv{B}}$.

Conversely, suppose that $\alpha_*$ is 1-1 on $X_\dv{B}$. Let $b$ be an atom of $\dv{B}$ and let $d = \bigwedge \{ \alpha(a) \mid b \le \alpha(a)\}$. Then $b \le d$. If $b < d$, there is an atom $c \ne b$ with $c \le d$. Therefore, $c \le \alpha(a)$ for each $a$ with $b \le \alpha(a)$, which implies that $\alpha_*({\uparrow}b) \subseteq \alpha_*({\uparrow}c)$ by Lemma~\ref{lem: atoms}. Since $\alpha_*({\uparrow}b)$ and $\alpha_*({\uparrow}c)$ are ends, we get $\alpha_*({\uparrow}b) = \alpha_*({\uparrow}c)$, and so $b=c$ since $\alpha_*$ is 1-1. This is a contradiction. Thus, $d=b$, and so $b$ is a meet from $\alpha[\dv{A}]$. Thus, $\alpha[\dv{A}]$ is join-meet dense in $\dv{B}$ by Remark~\ref{rem: jm density}(1), and hence $\alpha$ is a de Vries extension.
\end{proof}

\begin{definition} \label{def:tau_alpha}
For a de Vries extension $\alpha : \dv{A} \to \dv{B}$, define a topology $\tau_\alpha$ on $X_{\dv{B}}$ as the least topology making $\alpha_* : X_{\dv{B}} \to Y_\dv{A}$ continuous.
\end{definition}

\begin{remark}
Following usual terminology, we will refer to the topological space $(X_\dv{B}, \tau_\alpha)$ as $X_\dv{B}$ when there is no danger of confusion about which topology we are using.
\end{remark}

\begin{theorem} \label{lem:compactification}
If $\alpha : \dv{A} \to \dv{B}$ is a de Vries extension, then $X_\dv{B}$ is completely regular and $\alpha_* : X_{\dv{B}} \to Y_\dv{A}$ is a compactification.
\end{theorem}

\begin{proof}
By Lemma~\ref{lem:char of jm density}, $\alpha_* : X_\dv{B} \to Y_\dv{A}$ is 1-1. Therefore, $X_\dv{B}$ is homeomorphic to $\alpha_*[X_\dv{B}]$
by Definition~\ref{def:tau_alpha}, and so $X_\dv{B}$ is completely regular. Since $\alpha$ is 1-1, $\alpha_* : Y_\dv{B} \to Y_\dv{A}$ is onto, and $X_{\dv{B}}$ is dense in $Y_\dv{B}$
because $\dv{B}$ is atomic. Thus, $\alpha_*[X_\dv{B}]$ is dense in $Y_\dv{A}$, and hence $\alpha_* : X_{\dv{B}} \to Y_\dv{A}$ is a compactification.
\end{proof}

We are ready to define the functor ${\sf C} : \deve \to \C$. If $\alpha : \dv{A} \to \dv{B}$ is an object in $\deve$, we define its image under $\sf C$ to be the compactification $\alpha_* : X_{\dv{B}} \to Y_\dv{A}$. For a morphism $(\rho, \sigma)$ in $\deve$
\[
\xymatrix@C5pc{
\dv{A} \ar[r]^{\alpha} \ar[d]_{\rho} & \dv{B} \ar[d]^{\sigma} \\
\dv{A}' \ar[r]_{\alpha'} & \dv{B}'
}
\]
we define ${\sf C}(\rho, \sigma)$ as the pair $(\sigma_+, \rho_*)$, where $\sigma_+$ is the Tarski dual of $\sigma$ and $\rho_*$ is the de Vries dual of $\rho$.
\[
\xymatrix@C5pc{
X_{\dv{B'}} \ar[r]^{\alpha'_*} \ar[d]_{\sigma_+} & Y_\dv{A'} \ar[d]^{\rho_*} \\
X_{\dv{B}} \ar[r]_{\alpha_*} & Y_\dv{A}
}
\]
If $\sigma_*$ is the de Vries dual of $\sigma$, then it is straightforward to see that $\sigma_*({\uparrow}b) = {\uparrow}\sigma_+(b)$
for each atom $b$ of $\dv{B}$. This together with $\sigma\circ \alpha = \alpha' \star \rho$ and de Vries duality yield that
$\alpha_* \circ \sigma_+ = \rho_* \circ \alpha'_*$, so ${\sf C}(\rho, \sigma)$ is a morphism in $\C$.

\begin{proposition} \label{prop: functor C}
${\sf C} : {\deve} \to \C$ is a contravariant functor.
\end{proposition}

\begin{proof}
By Theorem~\ref{lem:compactification}, if $\alpha : \dv{A} \to \dv{B}$ is an object in $\deve$, then ${\sf C}(\alpha)$ is an object in $\C$,
and a morphism in $\deve$ yields a morphism in $\C$ by the discussion above. From the description of how $\sf C$ acts on morphisms, it is
clear that $\sf C$ sends identity morphisms to identity morphisms. Given a pair $(\rho_1,\sigma_1), (\rho_2,\sigma_2)$ of composable morphisms,
\[
\xymatrix@C5pc{
\dv{A}_1 \ar[r]^{\alpha_1} \ar[d]_{\rho_1} \ar@/_2pc/[dd]_{\rho_2 \star \rho_1} & \dv{B}_1 \ar[d]^{\sigma_1}
\ar@/^2pc/[dd]^{\sigma_2 \circ \sigma_1} \\
\dv{A}_2 \ar[r]^{\alpha_2} \ar[d]_{\rho_2} & \dv{B}_2 \ar[d]^{\sigma_2} \\
\dv{A}_3 \ar[r]_{\alpha_3} & \dv{B}_3
}
\]
$\sf C$ yields the following diagram
\[
\xymatrix@C3pc{
X_{\dv{B}_3}  \ar[r]^{(\alpha_3)_*} \ar[d]_{(\sigma_2)_+} \ar@/_4pc/[dd]_{(\sigma_1)_+ \circ (\sigma_2)_+} & Y_{\dv{A}_3} \ar[d]^{(\rho_2)_*}
\ar@/^4pc/[dd]^{(\rho_1)_* \circ (\rho_2)_*} \\
X_{\dv{B}_2} \ar[r]^{(\alpha_2)_*} \ar[d]_{(\sigma_1)_+} & Y_{\dv{A}_2} \ar[d]^{(\rho_1)_*} \\
X_{\dv{B}_1} \ar[r]_{(\alpha_1)_*} & Y_{\dv{A}_1}
}
\]
By de Vries duality, $(\rho_1)_* \circ (\rho_2)_* = (\rho_2 \star \rho_1)_*$, and
$(\sigma_2 \circ \sigma_1)_+ = (\sigma_1)_+ \circ (\sigma_2)_+$ by Tarski duality.
Therefore, ${\sf C}((\rho_2,\sigma_2) \circ (\rho_1,\sigma_1)) = {\sf C}(\rho_1,\sigma_1) \circ {\sf C}(\rho_2,\sigma_2)$.
Thus, $\sf C$ is a contravariant functor.
\end{proof}

\begin{theorem} \label{thm: duality}
The functors $\sf E$ and $\sf C$ yield a dual equivalence between $\C$ and $\deve$.
\end{theorem}

\begin{proof}
Propositions~\ref{prop: functor E} and \ref{prop: functor C} show that $\sf E$ and $\sf C$ are contravariant functors.  We first show that $\sf CE$ is naturally isomorphic to the identity functor on $\C$. The functor $\sf E$ sends $e : X \to Y$ to $e^{-1} : \RO(Y) \to \wp(X)$. Then $\sf C$ sends this de Vries extension to $(e^{-1})_* : X_{\wp(X)} \to Y_{\RO(Y)}$. We have the following diagram.
\[
\xymatrix@C4pc{
X \ar[r]^e \ar[d]_{\eta_X} & Y \ar[d]^{\xi_Y} \\
X_{\wp(X)} \ar[r]_{(e^{-1})_*} & Y_{\RO(Y)}
}
\]
The map $\xi_Y$ is a homeomorphism by de Vries duality, and $\eta_X$ is a bijection by Tarski duality.
The map $(e^{-1})_*$ sends $\{x\}$ to $\{ U \in \RO(Y) \mid e(x) \in U\}$, so the diagram is easily seen to be commutative.
Therefore, $\eta_X$ is also a homeomorphism by the definition of the topology on $X_{\wp(X)}$.

Let $(f,g)$ be a morphism in $\C$.
\[
\xymatrix@C5pc{
X \ar[r]^e \ar[d]_f & Y \ar[d]^g \\
X' \ar[r]_{e'} & Y'
}
\]
Then ${\sf E}(f,g) = (g^*, f^{-1})$, and so ${\sf CE}(f,g) = {\sf C}(g^*, f^{-1}) = ((f^{-1})_+, (g^*)_*)$. Thus, we have another morphism in $\C$.
\[
\xymatrix@C5pc{
X_{\wp(X)} \ar[r]^{(e^{-1})_*} \ar[d]_{(f^{-1})_+} & Y_{\RO(Y)} \ar[d]^{(g^*)_*} \\
X_{\wp(X')} \ar[r]_{((e')^{-1})_*} & Y_{\RO(Y')}
}
\]
We define a natural transformation $p$ from the identity functor to $\sf{CE}$ as follows. For a compactification $e : X \to Y$, we set $p_e = (\eta_X, \xi_Y)$. Since $\eta_X$ and $\xi_Y$ are homeomorphisms, $p_e$ is an isomorphism in $\C$.

We have the following diagram.
\[
\xymatrix@C5pc{
& X_{\wp(X)} \ar[rr]^{(e^{-1})_*} \ar[dd]^(.3){(f^{-1})_+} && Y_{\RO(Y)} \ar[dd]^{(g^*)_*} \\
X \ar[ru]^{\eta_X} \ar[dd]_{f} \ar[rr]^(.7){e} && Y \ar[dd]_(.3){g} \ar[ru]^{\xi_Y} & \\
& X_{\wp(X')}  \ar[rr]_(.3){((e')^{-1})_*} && Y_{\RO(Y')} \\
X'  \ar[rr]_{e'} \ar[ru]^{\eta_{X'}} && Y'  \ar[ru]^{\xi_{Y'}} &
}
\]
The front and back faces of this cube are commutative since $(f,g)$ and ${\sf CE}(f,g)$ are morphisms in $\C$. The top and bottom faces are commutative since $p_e$ and $p_{e'}$ are morphisms in $\C$. The left face is commutative by Tarski duality, and the right face is commutative by de Vries duality. Thus, $p$ is a natural isomorphism.

We next show that $\sf EC$ is naturally isomorphic to the identity functor on $\deve$. Given a de Vries extension $\alpha : \dv{A} \to \dv{B}$, the functor $\sf C$ sends it to $\alpha_* : X_{\dv{B}} \to Y_\dv{A}$. This is then sent by $\sf E$ to $\alpha_*^{-1} : \RO(Y_\dv{A}) \to \wp(X_{\dv{B}})$. We have the following diagram.
\[
\xymatrix@C4pc{
\dv{A} \ar[r]^{\alpha} \ar[d]_{\zeta_\dv{A}} & \dv{B} \ar[d]^{\vartheta_\dv{B}} \\
\RO(Y_\dv{A}) \ar[r]_{\alpha_*^{-1}} & \wp(X_{\dv{B}})
}
\]
The map $\zeta_\dv{A}$ is a de Vries isomorphism by de Vries duality and $\vartheta_\dv{B}$ is an isomorphism by Tarski duality.
To see that the diagram commutes,
\[
\vartheta_\dv{B}(\alpha(b)) = \vartheta_\dv{B}\left(\bigvee \{ \alpha(a) \mid a \prec b\}\right) =
\bigcup \{ \vartheta_\dv{B}(\alpha(a)) \mid a \prec b\}.
\]
On the other hand, $\alpha_*^{-1}(\zeta_\dv{A}(b)) = \bigcup \{ \alpha_*^{-1}(\zeta_\dv{A}(a)) \mid a \prec b\}$.
We have $\vartheta_\dv{B}(\alpha(a)) = \{ x \in X_\dv{B} \mid x \le \alpha(a) \}$. By Lemma~\ref{lem: atoms},
this is equal to $\alpha_*^{-1}(\zeta_\dv{A}(a))$. Thus, the diagram commutes, and so
$(\zeta_\dv{A}, \vartheta_\dv{B})$ is a morphism in $\deve$.

We define a natural transformation $q$ from the identity functor to $\sf{EC}$ as follows. For a de Vries extension $\alpha : \dv{A} \to \dv{B}$, we set $q_\alpha = (\zeta_\dv{A}, \vartheta_\dv{B})$, a morphism in $\deve$. Since $\zeta_\dv{A}$ and $\vartheta_\dv{B}$ are isomorphisms, $q_\alpha$ is an isomorphism in $\deve$.

To show naturality, let $(\rho, \sigma)$ be a morphism in $\deve$
\[
\xymatrix@C5pc{
\dv{A} \ar[r]^{\alpha} \ar[d]_{\rho} & \dv{B} \ar[d]^{\sigma} \\
\dv{A}' \ar[r]_{\alpha'} & \dv{B}'
}
\]
We have the following diagram.
\[
\xymatrix@C5pc{
& \RO(Y_\dv{A}) \ar[rr]^{(\alpha_*)^{-1}} \ar[dd]^(.3){(\rho_*)^*} && \wp(X_\dv{B}) \ar[dd]^{\sigma_+^{-1}} \\
\dv{A} \ar[ru]^{\zeta_\dv{A}} \ar[dd]_{\rho} \ar[rr]^(.7){\alpha} && \dv{B} \ar[dd]_(.3){\sigma} \ar[ru]^{\vartheta_\dv{B}} & \\
& \RO(Y_{\dv{A}'})  \ar[rr]_(.3){(\alpha'_*)^{-1}} && \wp(X_{\dv{B}'}) \\
\dv{A}'  \ar[rr]_{\alpha'} \ar[ru]^{\zeta_\dv{A'}} && \dv{B}'  \ar[ru]^{\vartheta_\dv{B'}} &
}
\]
The front and back faces of this cube are commutative because $(\rho, \sigma)$ and $\sf{EC}(\rho, \sigma)$ are morphisms in $\deve$. The top and bottom faces are commutative because $q_\alpha$ and $q_{\alpha'}$ are morphisms in $\deve$. The left face is commutative by de Vries duality and the right face is commutative by Tarski duality.
Therefore, $q$ is a natural isomorphism. Thus, $\sf E$ and $\sf C$ yield a dual equivalence between $\C$ and $\deve$.
\end{proof}

The use of both de Vries and Tarski dualities was critical in the proof of Theorem~\ref{thm: duality}. In the next remark we indicate how these
two dualities can be viewed as special cases of our duality.

\begin{remark} \label{dualities remark}
\begin{enumerate}
\item[]
\item If $X \in \KHaus$, then the identity $X \to X$ is the only compactification of $X$. The corresponding de Vries extension is
$\RO(X) \hookrightarrow \wp(X)$. In \cite[Sec.~5]{BMO18b} we will characterize such extensions,
and we will show that $\dev$ is equivalent to a full subcategory of $\deve$. It is also straightforward that $\KHaus$ is equivalent to a
full subcategory of $\C$. This leads to the following commutative diagram, where the functors $\KHaus \leftrightarrow \dev$ are those
of de Vries duality, and those between $\C$ and $\deve$ are $\sf E$ and $\sf C$.
\[
\xymatrix@C5pc{
\KHaus \ar@{<->}[r] \ar@{^{(}->}[d] & \dev \ar@{^{(}->}[d] \\
\C \ar@{<->}[r] & \deve
}
\]
Consequently, the duality of Theorem~\ref{thm: duality} is an extension of de Vries duality.
\item There is a full and faithful functor ${\sf Set} \to \C$ which sends $X$ to the Stone-\v{C}ech compactification $s : X \to \beta X$,
where $X$ is viewed as a discrete space. Since $\beta X$ is extremally disconnected, $\RO(X)$ is isomorphic to the Boolean algebra of clopen
subsets of $\beta X$, which in turn is isomorphic to $\wp(X)$. Therefore, the de Vries extension $s^{-1}:\RO(\beta X)\to\wp(X)$ is an isomorphism
of Boolean algebras. There is also a full and faithful functor ${\sf CABA} \to \deve$ which sends $\dv{B}$ to the identity
map $\dv{B} \to \dv{B}$. Thus, $\sf Set$ is equivalent to a full subcategory of $\C$, $\sf CABA$ is equivalent to a full subcategory
of $\deve$, and the following diagram is commutative, where the functors ${\sf Set} \leftrightarrow {\sf CABA}$ are those of Tarski duality.
\[
\xymatrix@C5pc{
{\sf Set} \ar@{<->}[r] \ar@{^{(}->}[d] & {\sf CABA} \ar@{^{(}->}[d] \\
\C \ar@{<->}[r] & \deve
}
\]
Consequently, the duality of Theorem~\ref{thm: duality} is an extension of Tarski duality.
\end{enumerate}
\end{remark}

\section{Maximal de Vries extensions and Stone-\v{C}ech compactifications} \label{sec: maximal}

In this final section we introduce maximal de Vries extensions and prove that they correspond to Stone-\v{C}ech compactifications. From this we derive that the category $\creg$ of completely regular spaces and continuous maps is dually equivalent to the full subcategory $\Mdeve$ of $\deve$ consisting of maximal extensions.

Let
\[
\xymatrix{
X \ar[r]^{e} \ar[rd]_{e'} & Y \ar[d]^{f} \\
& Y'
}
\]
be a commutative diagram of compactifications of $X$. The functor $\sf E$ sends this diagram to the commutative diagram
\[
\xymatrix{
\RO(Y')  \ar[d]_{f^*} \ar[r]^{(e')^{-1}} & \wp(X) \\
\RO(Y) \ar[ru]_{e^{-1}}
}
\]
By Lemma~\ref{lem:RO(X)}, $e^{-1}[\RO(Y)] = \RO(X) = (e')^{-1}[\RO(Y')]$. This motivates the following definition.

\begin{definition}
We call two de Vries extensions $\alpha : \dv{A} \to \dv{B}$ and $\gamma : \dv{C} \to \dv{B}$ \emph{compatible} if $\alpha[\dv{A}]=\gamma[\dv{C}]$.
\end{definition}

\begin{lemma} \label{lem: homeomorphism}
Two de Vries extensions $\alpha : \dv{A} \to \dv{B}$ and $\gamma : \dv{C} \to \dv{B}$ are compatible iff the topologies $\tau_\alpha$ and $\tau_\gamma$ on $X_\dv{B}$ given in Definition~\ref{def:tau_alpha} are equal.
\end{lemma}

\begin{proof}
First suppose that $\alpha : \dv{A} \to \dv{B}$ and $\gamma : \dv{C} \to \dv{B}$ are compatible, so $\alpha[\dv{A}] = \gamma[\dv{C}]$. By definition,
$\tau_\alpha$ is the least topology making $\alpha_*:X_{\dv{B}}\to Y_\dv{A}$ continuous and $\tau_\gamma$ is the least topology making $\gamma_*:X_{\dv{B}}\to Y_\dv{C}$ continuous. Since $\zeta[\dv{A}]$ is a basis for the topology on $Y_\dv{A}$, we have that $\alpha_*^{-1}(\zeta[\dv{A}])$ is a basis for $\tau_\alpha$, and similarly $\gamma_*^{-1}(\zeta[\dv{C}])$ is a basis for $\tau_\gamma$. Therefore, it is sufficient to show that $\alpha_*^{-1}(\zeta[\dv{A}])=\gamma_*^{-1}(\zeta[\dv{C}])$. But this follows from $\alpha[\dv{A}] = \gamma[\dv{C}]$ by Lemma~\ref{lem: atoms}.

For the converse, suppose that $\tau_\alpha=\tau_\gamma$. By Lemma~\ref{lem:RO(X)}, $\alpha_*^{-1}:\RO(Y_\dv{A}) \to \RO(X_{\dv{B}})$ and $\gamma_*^{-1}:\RO(Y_\dv{C}) \to \RO(X_{\dv{B}})$ are Boolean isomorphisms. By de Vries duality, $\zeta[\dv{A}] = \RO(Y_\dv{A})$ and $\zeta[\dv{C}] = \RO(Y_\dv{C})$. Therefore, since $\tau_\alpha=\tau_\gamma$, we have $\alpha_*^{-1}(\zeta[\dv{A}])=\gamma_*^{-1}(\zeta[\dv{C}])$. By Lemma~\ref{lem: atoms}, $\alpha[\dv{A}]=\gamma[\dv{C}]$, so $\alpha : \dv{A} \to \dv{B}$ and $\gamma : \dv{C} \to \dv{B}$ are compatible.
\end{proof}

\begin{definition}
\begin{enumerate}
\item[]
\item We say that a de Vries extension $\alpha : \dv{A} \to \dv{B}$ is \emph{maximal} provided for every compatible de Vries extension
$\gamma : \dv{C} \to \dv{B}$, there is a de Vries morphism $\delta: \dv{C} \to \dv{A}$ such that $\alpha \star \delta = \gamma$.
\[
\xymatrix{
\dv{A} \ar[rr]^{\alpha} && \dv{B} \\
& \dv{C} \ar[lu]^{\delta} \ar[ru]_{\gamma} &
}
\]
\item Let $\Mdeve$ be the full subcategory of $\deve$ consisting of maximal de Vries extensions.
\end{enumerate}
\end{definition}

\begin{theorem} \label{thm: Stone-Cech}
Let $e : X \to Y$ be a compactification and let $s : X \to \beta X$ be the Stone-\v{C}ech compactification of $X$. Then the following conditions are equivalent.
\begin{enumerate}
\item The de Vries extension $e^{-1}: \RO(Y) \to \wp(X)$ is maximal.
\item $e$ is isomorphic to $s$ in $\C$.
\item $e$ is equivalent to $s$.
\end{enumerate}
\end{theorem}

\begin{proof}
(1)$\Rightarrow$(2).
First suppose that $e^{-1}: \RO(Y) \to \wp(X)$ is a maximal de Vries extension. By Theorem~\ref{prop: extension of a compactification}, the Stone-\v{C}ech compactification $s:X\to\beta X$ yields a de Vries extension $s^{-1}: \RO(\beta X) \to \wp(X)$. By Lemma~\ref{lem:RO(X)}, $e^{-1}[\RO(Y)]=\RO(X)=s^{-1}[\RO(\beta X)]$. Therefore, $e^{-1}$ and $s^{-1}$ are compatible, so by maximality of $e^{-1}$, there is a
de Vries morphism $\delta: \RO(\beta X) \to \RO(Y)$ such that $e^{-1}\star\delta=s^{-1}$.
\[
\xymatrix{
\RO(Y) \ar[rr]^{e^{-1}} && \wp(X) \\
& \RO(\beta X) \ar[lu]^{\delta} \ar[ru]_{s^{-1}} &
}
\]
Since the Stone-\v{C}ech compactification is the largest compactification, we also have a continuous map $f:\beta X\to Y$ such that $f\circ s=e$.
\[
\xymatrix@C3pc{
X \ar[r]^{s} \ar[rd]_{e} & \beta X \ar[d]^{f} \\
& Y
}
\]
This induces a commutative diagram in $\dev$
\[
\xymatrix{
\RO(\beta X) \ar[rr]^{s^{-1}} && \wp(X) \\
& \RO(Y) \ar[lu]^{f^*} \ar[ru]_{e^{-1}} &
}
\]
From this it follows that $\delta,f^*$ are inverses of each other in $\dev$. Therefore, $f^*: \RO(Y) \to \RO(\beta X)$ is an isomorphism. Thus, by de Vries duality, $f:\beta X\to Y$ is a homeomorphism, and hence $(\func{id}_X, f)$ is an isomorphism from  $e:X\to Y$ to $s:X\to\beta X$.

(2)$\Rightarrow$(3). This follows from Theorem~\ref{prop: equivalent to beta}.

(3)$\Rightarrow$(1). Suppose that $e:X\to Y$ is equivalent to $s:X\to\beta X$. Then there is a homeomorphism $f:\beta X\to Y$ such that $f\circ s=e$. But then $f^*: \RO(Y) \to \RO(\beta X)$ is an isomorphism such that $s^{-1}\star f^*=e^{-1}$. Let $\gamma : \dv{C} \to \wp(X)$ be a de Vries extension compatible with $e^{-1}: \RO(Y) \to \wp(X)$. Then $\gamma$ is compatible with $s^{-1} : \RO(\beta X) \to \wp(X)$, so by Lemma~\ref{lem: homeomorphism}, the topology on $X$ is $\tau_\gamma$, and hence $\gamma_* : X \to Y_\dv{C}$ is a compactification. Therefore, there is a continuous $g:\beta X\to Y_\dv{C}$ such that $g\circ s=\gamma_*$.
\[
\xymatrix@C3pc{
X \ar[r]^{s} \ar[rd]_{\gamma_*} & \beta X \ar[d]^{g} \\
& Y_\dv{C}
}
\]
This yields a commutative diagram in $\dev$
\[
\xymatrix{
\RO(\beta X) \ar[rr]^{s^{-1}} && \wp(X) \\
& \dv{C} \ar[lu]^{g^* \star \zeta_\dv{C}} \ar[ru]_{\gamma} &
}
\]
Since $f^*: \RO(Y) \to \RO(\beta X)$ is an isomorphism such that $s^{-1}\star f^*=e^{-1}$, this in turn induces a commutative diagram in $\dev$
\[
\xymatrix{
\RO(Y) \ar[rr]^{e^{-1}} && \wp(X) \\
& \dv{C} \ar[lu]^{(f^*)^{-1}\star g^* \star \zeta_\dv{C}} \ar[ru]_{\gamma} &
}
\]
Thus, $e^{-1} : \RO(Y) \to \wp(X)$ is a maximal de Vries extension.
\end{proof}

\begin{remark} \label{rem: Stone-Cech}
Due to Theorem~\ref{thm: duality}, we can equivalently formulate Theorem~\ref{thm: Stone-Cech} in terms of de Vries extensions as follows. Let $\alpha : \dv{A} \to \dv{B}$ be a de Vries extension. Then the following are equivalent.
\begin{enumerate}
\item $\alpha$ is maximal,
\item  $\alpha_* : X_\dv{B} \to Y_\dv{A}$ is isomorphic in $\C$ to the Stone-\v{C}ech compactification $s : X_\dv{B} \to \beta X_\dv{B}$.
\item $\alpha_*$ is equivalent to the Stone-\v{C}ech compactification $s$.
\end{enumerate}
\end{remark}

Theorem~\ref{thm: Stone-Cech} motivates the following definition.

\begin{definition}
Let $\SC$ be the full subcategory of $\C$ consisting of Stone-\v{C}ech compatifications.
\end{definition}

Theorem~ \ref{thm: duality} and Remark~\ref{rem: Stone-Cech} immediately yield the following:

\begin{theorem} \label{cor: MDeVe = MComp}
There is a dual equivalence between $\Mdeve$ and $\SC$.
\end{theorem}

\begin{proposition} \label{lem: CReg = MComp}
There is an equivalence between $\creg$ and $\SC$.
\end{proposition}

\begin{proof}
Define $\sf{B} : \creg \to \SC$ by sending a completely regular space $X$ to the Stone-\v{C}ech compactification $s_X : X \to \beta X$ and a continuous map $f : X \to Y$ to $(f,\beta f)$:
\[
\xymatrix@C5pc{
X \ar[r]^{s_X} \ar[d]_{f} & \beta X \ar[d]^{\beta f} \\
Y \ar[r]_{s_Y} & \beta Y
}
\]
The universal mapping property of $\beta $ shows that this diagram is commutative, so $(f, \beta f)$ is a morphism in $\C$. If $f : X \to Y$ and $g : Y \to Z$ are morphisms in $\creg$, then since $\beta:\creg\to\KHaus$ is a functor,
\[
{\sf B}(g\circ f) = (g\circ f, \beta(g\circ f)) = (g\circ f, \beta g\circ\beta f) = (g,\beta g)\circ(f, \beta f) = {\sf B}(g)\circ{\sf B}(f),
\]
which shows that $\sf{B}$ preserves composition. Because it is clear that $\sf{B}$ sends identity maps to identity maps, $\sf{B}$ is a functor.

It is obvious that $\sf{B}$ is faithful. To see that $\sf{B}$ is full, let $(f, g)$ be a morphism between Stone-\v{C}ech compactifications:
\[
\xymatrix@C5pc{
X \ar[r]^{s_X} \ar[d]_{f} & \beta X \ar[d]^{g} \\
Y \ar[r]_{s_Y} & \beta Y
}
\]
Then $g$ is a continuous map satisfying $g \circ s_X = s_Y \circ f$. Since $\beta f$ satisfies
$\beta f\circ s_X = s_Y \circ f$, we have that $g$ and $\beta f$ agree on the dense subspace $s_X(X)$ of $\beta X$, and so $g = \beta f$. Therefore, $(f,g) = {\sf B}(f)$, showing that $\sf{B}$ is full. Thus, $\creg$ is equivalent to $\SC$ \cite[Thm.~IV.4.1]{Mac71}.
\end{proof}

We are ready to generalize de Vries duality to the category $\creg$ of completely regular spaces and continuous maps.

\begin{theorem}\label{thm: comp reg}
There is a dual equivalence between $\creg$ and $\Mdeve$.
\end{theorem}

\begin{proof}
This follows from Theorem~\ref{cor: MDeVe = MComp} and Proposition~\ref{lem: CReg = MComp}.
\end{proof}

In the sequel \cite{BMO18b} we will show how to use Theorems~\ref{thm: duality} and \ref{thm: comp reg} to characterize such topological properties as normality and local compactness in terms of de Vries extensions that are subject to additional axioms.

\def\cprime{$'$}
\providecommand{\bysame}{\leavevmode\hbox to3em{\hrulefill}\thinspace}
\providecommand{\MR}{\relax\ifhmode\unskip\space\fi MR }
\providecommand{\MRhref}[2]{%
  \href{http://www.ams.org/mathscinet-getitem?mr=#1}{#2}
}
\providecommand{\href}[2]{#2}

\bigskip

\noindent Department of Mathematical Sciences, New Mexico State University, Las Cruces NM 88003

\bigskip

\noindent guram@nmsu.edu, pmorandi@nmsu.edu, olberdin@nmsu.edu

\end{document}